\documentclass[11pt,reqno]{amsart}

\usepackage[T1]{fontenc}
\usepackage[utf8]{inputenc}
\usepackage[a4paper,margin=2.8cm]{geometry}
\usepackage{amsmath,amssymb,amsfonts,amsthm,mathtools}
\usepackage{enumitem}
\usepackage{microtype}
\usepackage[hidelinks,bookmarksdepth=3]{hyperref}

\numberwithin{equation}{section}

\theoremstyle{plain}
\newtheorem{theorem}{Theorem}[section]
\newtheorem{lemma}[theorem]{Lemma}
\newtheorem{proposition}[theorem]{Proposition}
\newtheorem{corollary}[theorem]{Corollary}

\theoremstyle{definition}

\newtheorem{remark}[theorem]{Remark}

\usepackage[nameinlink,noabbrev]{cleveref}

\crefname{theorem}{Theorem}{Theorems}
\crefname{lemma}{Lemma}{Lemmas}
\crefname{proposition}{Proposition}{Propositions}
\crefname{corollary}{Corollary}{Corollaries}
\crefname{remark}{Remark}{Remarks}

\newcommand{\Prb}{\mathbb P}
\newcommand{\E}{\mathbb E}
\newcommand{\calC}{\mathcal C}
\newcommand{\calH}{\mathcal H}
\newcommand{\calG}{\mathcal G}
\newcommand{\calI}{\mathcal I}
\newcommand{\ind}{\mathrm{ind}}

\title[A sparse transference principle for a non-monotone Ramsey property]{A sparse transference principle \\for a non-monotone Ramsey property}
\author[G. Carenini]{Gaia Carenini}
\address{Trinity College Cambridge, Department of Pure Mathematics and Mathematical Statistics, Centre for Mathematical Sciences, Wilberforce Road, Cambridge CB3 0WA, United Kingdom.}
\email{gc645@cam.ac.uk}
\date{\today}

\begin{document}

\begin{abstract}
We prove a sparse transference theorem for induced Ramsey graphs.  The theorem
transfers the weighted random-host proof of Arag\~ao, Campos, Dahia, Filipe
and Marciano to the sparse random setting.  It follows that, for every fixed
graph $H$ with no isolated vertices and at least two edges, and every
$\eta>0$, there is $C>0$ such that, whenever
$N\ge r^{Cr}$ and $N^{-1/m_2(H)+\eta}\le p\le \frac12$, with high probability every $r$-colouring of the edges of $G(N,p)$ contains
a monochromatic induced copy of $H$.  Here $m_2(H)$ denotes the usual
maximum 2-density of $H$.
\end{abstract}

\maketitle
\vspace{-0.8cm}
\section{Introduction}
For graphs $F,H$ and an integer $r\ge2$, write $ F\longrightarrow_{\ind}(H)_r$ if every $r$-colouring of the edges of $F$ contains a monochromatic
induced copy of $H$.  The induced Ramsey number $R_{\ind}(H;r)$ is the
least integer $N$ for which there is an $N$-vertex graph $F$ with
$F\to_{\ind}(H)_r$.  A recent theorem of Arag\~ao, Campos, Dahia, Filipe and
Marciano gives the exponential bound
$$
        R_{\ind}(H;r)\le r^{Crh}
$$
for every graph $H$ on $h$ vertices \cite{ACDFM}.  Their proof is random in
nature: at the scale $M=r^{O_H(r)}$, a typical graph $G(M,1/2)$ is
induced-Ramsey for $H$ with $r$ colours.  More importantly for us, their
argument proves a stronger weighted statement.  In every colouring of such a
typical dense graph, some colour contains induced copies of $H$ supporting a
measure with small weighted overlaps. More precisely, these copies form a hypergraph admitting a non-zero measure
$\nu$ for which the weighted sum of squared codegrees
$$
        \sum_{\substack{S\subseteq V\\ |S|\ge2}}
        d_\nu(S)^2\rho^{-|S|}
$$
is small compared with $e(\nu)^2$, where $d_\nu(S)$ denotes the
$\nu$-mass of copies containing $S$, and $e(\nu)$ is the total mass of
$\nu$.  This is the Janson condition of Arag\~ao, Campos, Dahia, Filipe and
Marciano; we recall the precise form in \Cref{sec:weighted-input}.

The purpose of this note is to show that this extra distributional information
has a consequence in the sparse setting.  The idea of the proof is as follows.  The bare
assertion that $G(M,1/2)$ is induced-Ramsey would only allow one to use
$M$-vertex subgraphs in which no edge has been lost.  The weighted statement
is stronger: its pair-overlap consequence implies resilience under the deletion
of many edges.  Once this resilience is averaged over the $M$-vertex induced
subgraphs of a larger random graph $G(N,1/2)$, it gives a supersaturation
statement for tuples of edge sets.  Hypergraph containers then compress the
possible colour classes, and a random thinning of $G(N,1/2)$ hits the
complement of every bad tuple.

Throughout the paper $H$ is fixed, has no isolated vertices, and has
$e(H)\ge2$.  Put $ h=v(H)$, $ q=e(H)$, and define
$$
        m_2(H)=
        \max_{\substack{J\subseteq H\\ v(J)\ge3}}
        \frac{e(J)-1}{v(J)-2}.
$$
This is the usual maximum 2-density in sparse Ramsey theory.  All
high-probability statements in this paper mean that the relevant probability
tends to one as $N\to\infty$.  Our main result is the following.

\begin{theorem}[Sparse random induced Ramsey theorem]
\label{thm:main}
Let $H$ be fixed, with no isolated vertices and with $e(H)\ge2$.  For every
$\eta>0$, there is a constant $C>0$ such that the following holds.
For every $r\ge2$, if $ N\ge r^{Cr}$ and $ N^{-1/m_2(H)+\eta}\le p\le \frac12$,
then $ G(N,p)\longrightarrow_{\ind}(H)_r$ with high probability.
\end{theorem}

The upper bound $p\le1/2$ is part of the formulation, not a technical
artefact.  It reflects the non-monotonicity discussed above: for non-complete
$H$, the conclusion cannot be extended to all larger values of $p$.
Equivalently, the theorem says that at a slightly enlarged version of the
vertex scale supplied by \cite{ACDFM}, the dense edge exponent $2$ may be
replaced by the sparse exponent $2-1/m_2(H)$, up to an arbitrary loss.

As a direct consequence we obtain the following induced size-Ramsey bound.  Let
$\widehat r_{\ind}(H;r)$ denote the minimum number of edges in a graph
$G$ satisfying $G\to_{\ind}(H)_r$.

\begin{corollary}[Size-Ramsey consequence]
\label{cor:size}
For every $\eta>0$,
$$
        \widehat r_{\ind}(H;r)
        \le r^{Cr(2-1/m_2(H)+\eta)}.
$$
Equivalently, for $N=r^{Cr}$ one can find an $N$-vertex graph
$G$ with $G\to_{\ind}(H)_r$ and
$$
        e(G)\le N^{2-1/m_2(H)+\eta}.
$$
\end{corollary}
\medskip
\section{The weighted dense random input}
\label{sec:weighted-input}

We begin by recalling the precise form of the weighted input needed from
\cite{ACDFM}.  The notation below is equivalent to the notation used there, but
is specialized to the present application.

Let $T$ be a graph and let $G\subseteq T$.  We write
$\calI_H(G,T)$ for the $h$-uniform hypergraph with vertex set $V(T)$ whose
hyperedges are the sets $U\in\binom{V(T)}h$ such that $T[U]=G[U]$ is
isomorphic to $H$.  Thus the hyperedges are induced copies of $H$ in
$T$ whose graph-edges all lie in $G$.

Let $\calG$ be a hypergraph on vertex set $V$.  If $\nu$ is a
non-negative measure on $E(\calG)$, write
$$
        e(\nu)=\sum_{Q\in E(\calG)}\nu(Q),
$$
and, for $S\subseteq V$,
$$
        d_\nu(S)=
        \sum_{\substack{Q\in E(\calG)\\ S\subseteq Q}}\nu(Q).
$$
For $0<\rho\le1$, set
$$
        \Lambda_\rho(\nu)=
        \sum_{\substack{S\subseteq V\\ |S|\ge2}}
        d_\nu(S)^2\rho^{-|S|}.
$$
We say that $\calG$ is $(\rho,R)$-Janson if there is a non-zero
non-negative measure $\nu$ on $E(\calG)$ such that
$$
        \Lambda_\rho(\nu)<\frac{e(\nu)^2}{R}.
$$
This condition is homogeneous in $\nu$, so $\nu$ may be normalized to have
$e(\nu)=1$.

\begin{theorem}[Weighted random-host input, extracted from \cite{ACDFM}]
\label{thm:acdfm-input}
Let $H$ be fixed, with $h=v(H)$, and let $r\ge2$.  Put $\rho=2^{-25h^2}r^{-4}$.
There are constants $A_H,b_H,B_H>0$ such that, whenever $M\ge r^{A_Hr}$, the graph
$T\sim G(M,1/2)$ has the following property with probability at least
$$
        1-\exp\{-b_H r^{-B_Hr}M^2\}.
$$
For every colouring $c:E(T)\to[r]$, there is a colour $a\in[r]$ such that,
writing $T_a$ for the colour-$a$ subgraph of $T$, the hypergraph $ \calI_H(T_a,T)$ is $(\rho,\rho M)$-Janson.  In particular, after increasing $A_H$ if
necessary, the probability is at least $0.99$.
\end{theorem}

\begin{proof}[Extraction from \cite{ACDFM}]
We spell out the reduction because the statement used here is stronger than
the headline induced Ramsey bound.  In the notation of \cite{ACDFM}, Definition
2.1 defines, for a graph $F$ and graphs $G'\subseteq G$, the hypergraph
$\mathcal J_{F,G',G}$ with vertex set $V(G)$, whose hyperedges are the sets
$L\subseteq V(G)$ such that $F\cong G'[L]=G[L]$.  Thus, if $G'$ is a
colour class in an edge-colouring of $G$, this is precisely our hypergraph of
monochromatic copies of $F$ which are induced in the underlying graph $G$.
Their Definition 2.2 says that a hypergraph $\mathcal G$ is $(p,R)$-Janson
if there is a non-zero measure $\nu:E(\mathcal G)\to\mathbb R_{\ge0}$ with
$$
        \Lambda_p(\nu)<\frac{e(\nu)^2}{R},
$$
where $\Lambda_p$ is the quantity displayed above.  Observation 2.4 in
\cite{ACDFM} allows this measure to be normalized to have any prescribed
positive total mass.

In Section 2.2 of \cite{ACDFM} the authors fix $C=300$, $\delta=r^{-50}$, and $p=p(r,k)=2^{-25k^2}r^{-4}$.
Definition 2.7 defines the bad event $\mathcal B(\mathbf H)$, where
$\mathbf H=(H_i)_{i\in[r]}$, as the event that there exists an edge-colouring
of $G$ for which, for every colour $i$, the hypergraph
$\mathcal J_{H_i,G_i,G}$ is not $(p,pv(G))$-Janson.  When
$H_1=\cdots=H_r=H$, the complement of $\mathcal B(\mathbf H)$ is exactly
the assertion that every $r$-colouring has a colour $a$ for which
$\calI_H(T_a,T)$ is $(\rho,\rho M)$-Janson, with $\rho=p(r,h)$.

Lemma 2.8 of \cite{ACDFM}, together with Claim 2.9 and the union-bound
calculation following it, gives for $G\sim G(M,1/2)$ a failure probability of the form
$$
        \Prb\bigl(G\in\mathcal B(H;r)\bigr)
        \le \exp\{-b_H r^{-B_Hr}M^2\}
$$
provided $M\ge r^{Ahr}$ for a sufficiently large absolute constant $A$,
where $b_H,B_H>0$ depend only on $H$.  Writing $A_H=Ah$ gives the
displayed statement.
\end{proof}

The following elementary consequence is the only part of the Janson condition
used later.

\begin{remark}[Pair-overlap consequence]
\label{lem:pair-overlap}
Let $\calG$ be an $h$-uniform hypergraph on vertex set $V$, and suppose
that $\calG$ is $(\rho,\rho M)$-Janson.  Then there is a probability
measure $\nu$ on $E(\calG)$ such that, for
$$
        d_\nu(xy)=
        \sum_{\substack{Q\in E(\calG)\\ \{x,y\}\subseteq Q}}\nu(Q),
        \qquad xy\in\binom V2,
$$
one has
$$
        \sum_{xy\in\binom V2} d_\nu(xy)^2<\frac{\rho}{M}.
$$
\end{remark}

\begin{proof}
Choose $\nu$ with $e(\nu)=1$ and
$$
        \Lambda_\rho(\nu)<\frac{1}{\rho M}.
$$
The terms with $|S|=2$ are among the non-negative terms in
$\Lambda_\rho(\nu)$.  Hence
$$
        \rho^{-2}\sum_{xy\in\binom V2}d_\nu(xy)^2
        \le \Lambda_\rho(\nu)
        <\frac{1}{\rho M}.
$$
Multiplying by $\rho^2$ gives the claim.
\end{proof}

\section{Resilience and supersaturation}

The first step in the transference argument is that the pair-overlap estimate
makes the dense graph robust under edge deletion.

\begin{lemma}[Resilience under edge deletion]
\label{lem:resilience}
Let $T$ be an $M$-vertex graph satisfying the conclusion of
Theorem~\ref{thm:acdfm-input}, and let $L\subseteq E(T)$.  If
$$
        |L|\le \frac{M}{4\rho},
$$
then every $r$-colouring of $E(T)\setminus L$ contains a monochromatic
induced copy of $H$.
\end{lemma}

\begin{proof}
Extend the colouring of $E(T)\setminus L$ arbitrarily to a colouring of all
of $E(T)$.  By Theorem~\ref{thm:acdfm-input} and Lemma~\ref{lem:pair-overlap}, there is
a colour $a$ and a probability measure $\nu$ on colour-$a$ induced copies
of $H$ in $T$ such that
$$
        \sum_{xy\in\binom{V(T)}2}d_\nu(xy)^2<\frac{\rho}{M}.
$$
The total $\nu$-mass of copies using at least one edge of $L$ is at most
$$
        \sum_{xy\in L}d_\nu(xy).
$$
By Cauchy's inequality,
$$
        \sum_{xy\in L}d_\nu(xy)
        \le |L|^{1/2}
        \left(\sum_{xy\in L}d_\nu(xy)^2\right)^{1/2}
        \le \left(|L|\frac{\rho}{M}\right)^{1/2}
        \le \frac12.
$$
Thus positive $\nu$-mass remains on colour-$a$ induced copies avoiding
$L$.  Such a copy is induced in $T$, and all its graph-edges lie in
$E(T)\setminus L$, so it is a monochromatic induced copy in the deleted
graph.
\end{proof}

We now average this statement over $M$-vertex subsets of a larger dense
random graph.

\begin{lemma}[Supersaturation in the dense random host]
\label{lem:supersaturation}
Let $M\ge r^{A_Hr}$, let $N\ge M$, and let $F\sim G(N,1/2)$.  With
probability at least $1-\exp\{-cN^2/M^4\}$, 
where $c>0$ is absolute, the following holds.  For every
$C_1,\ldots,C_r\subseteq E(F)$, if
$$
        |E(F)\setminus(C_1\cup\cdots\cup C_r)|
        \le c_H\frac{N^2}{\rho M},
$$
then some $C_a$ contains at least
$$
        c_H\frac{N^h}{rM^h}
$$
labelled induced copies of $H$ in $F$.  Here $c_H>0$ depends only on
$H$.
\end{lemma}

\begin{proof}
Call a set $S\in\binom{V(F)}M$ good if $F[S]$ satisfies the conclusion of
Theorem~\ref{thm:acdfm-input}.  For each fixed $S$, the graph $F[S]$ has
distribution $G(M,1/2)$, so
$$
        \E |\{S:S\text{ is good}\}|
        \ge .99\binom NM.
$$
Changing one edge of $F$ affects at most $\binom{N-2}{M-2}$ of the
indicators $1_{\{S\text{ good}\}}$.  The bounded-differences inequality
therefore gives
$$
        \Prb\left[|\{S:S\text{ is good}\}|<.9\binom NM\right]
        \le \exp\{-cN^2/M^4\}
$$
for an absolute constant $c>0$.  Condition on the complementary event.

Let $L=E(F)\setminus(C_1\cup\cdots\cup C_r)$ and assume
$$
        |L|\le c_H\frac{N^2}{\rho M}.
$$
Call $S\in\binom{V(F)}M$ light if
$$
        |L\cap E(F[S])|\le \frac{M}{4\rho}.
$$
Averaging over $S$ gives
$$
        \E_S |L\cap E(F[S])|
        = |L|\frac{\binom{N-2}{M-2}}{\binom NM}
        \le 2|L|\frac{M^2}{N^2}
        \le 2c_H\frac{M}{\rho}.
$$
If $c_H$ is chosen small enough, Markov's inequality
implies that at least $.9\binom NM$ sets are light.  Hence at least
$.8\binom NM$ sets are both good and light.

Fix such an $S$.  Colour every edge of $F[S]\setminus L$ by one colour
$a$ for which it lies in $C_a$.  Since $S$ is good and light,
Lemma~\ref{lem:resilience} gives a monochromatic induced copy of $H$ in
$F[S]\setminus L$.  If its colour is $a$, then all graph-edges of the copy
lie in $C_a$.

Choose one labelled copy obtained in this way for every good light set $S$.
A fixed labelled induced copy of $H$ is contained in at most
$\binom{N-h}{M-h}$ choices of $S$.  If $X_a$ is the number of labelled
induced copies of $H$ in $F$ whose graph-edges all lie in $C_a$, then
$$
        \sum_{a=1}^r X_a
        \ge
        \frac{.8\binom NM}{\binom{N-h}{M-h}}
        \ge c_H\frac{N^h}{M^h}.
$$
Thus some $a$ satisfies $X_a\ge c_HN^h/(rM^h)$, as required.
\end{proof}

We shall use the following contrapositive form.

\begin{corollary}[Tuple supersaturation]
\label{cor:tuple-supersaturation}
Under the hypotheses of Lemma~\ref{lem:supersaturation}, with probability at least
$1-\exp\{-cN^2/M^4\}$, the graph $F\sim G(N,1/2)$ has the following
property.  If $C_1,
\ldots,C_r\subseteq E(F)$ and each $C_a$ contains fewer than
$c_HN^h/(rM^h)$ labelled induced copies of $H$, then
$$
        |E(F)\setminus(C_1\cup\cdots\cup C_r)|
        \ge c_H\frac{N^2}{\rho M}.
$$
\end{corollary}

\section{Containers for induced-copy-free edge sets}

Let $F$ be a graph on $N$ vertices.  Define a $q$-uniform hypergraph
$\calH_F$ as follows.  The vertex set of $\calH_F$ is $E(F)$.  The
hyperedges of $\calH_F$ are the edge sets of induced copies of $H$ in
$F$.  An edge set $I\subseteq E(F)$ is independent in $\calH_F$ precisely
when $I$ contains no induced copy of $H$ whose non-edges are also non-edges
of $F$.

We use the following standard form of the hypergraph container theorem.

\begin{theorem}[Hypergraph containers]
\label{thm:containers}
For every integer $q\ge2$ and every $A\ge1$, there are constants
$C=C(q,A)>0$ and $\tau_0=\tau_0(q,A)>0$ with the following property.  Let
$\calG$ be a $q$-uniform hypergraph on $n$ vertices with average degree
$d$.  Suppose that $0<\tau<\tau_0$ and
$$
        \Delta_j(\calG)\le A d\tau^{j-1}
        \qquad\text{for every }2\le j\le q.
$$
Then, for every $0<\alpha<1$, there is a family
$\calC\subseteq2^{V(\calG)}$ such that:
\begin{enumerate}[label=\textup{(\arabic*)}]
\item every independent set of $\calG$ is contained in some $C\in\calC$;
\item every $C\in\calC$ satisfies $e(\calG[C])\le \alpha e(\calG)$;
\item
$\log |\calC|\le Cn\tau\log(1/\tau)\log(1/\alpha)$.

\end{enumerate}
\end{theorem}

\begin{lemma}[Codegrees in a dense random host]
\label{lem:codegrees}
With high probability, $F\sim G(N,1/2)$ satisfies:
\begin{enumerate}[label=\textup{(\arabic*)}]
\item $|E(F)|=\Theta(N^2)$;
\item $e(\calH_F)=\Theta_H(N^h)$;
\item the average degree of $\calH_F$ is $d=\Theta_H(N^{h-2})$;
\item if $\tau=N^{-1/m_2(H)}$, then, there exists a constant $A_H$,
$ \Delta_\ell(\calH_F)\le A_H d\tau^{\ell-1}$ for every $2\le \ell\le q$.
\end{enumerate}
\end{lemma}

\begin{proof}
The first two assertions follow from the usual concentration estimates for the
number of edges and for the number of induced copies of a fixed graph in
$G(N,1/2)$.  They imply the third, since
$$
        d=\frac{qe(\calH_F)}{|E(F)|}=\Theta_H(N^{h-2}).
$$

For the codegree estimate, fix a set $\sigma\subseteq E(F)$ of $\ell$ host
edges.  If $\sigma$ is contained in no induced copy of $H$, then its
codegree is zero.  Otherwise, in every induced copy containing $\sigma$, the
edges of $\sigma$ correspond to an $\ell$-edge subgraph $J\subseteq H$.
For a fixed such $J$, once the images of the vertices incident with these
$\ell$ edges are fixed, the remaining vertices of the copy can be chosen in
at most $O_H(N^{h-v(J)})$ ways.  Hence
$$
        \Delta_\ell(\calH_F)
        \le C_H\max_{\substack{J\subseteq H\\ e(J)=\ell}}N^{h-v(J)}.
$$
Since $\ell\ge2$, every contributing $J$ has $v(J)\ge3$.  By the
definition of $m_2(H)$, $\ell-1\le m_2(H)(v(J)-2)$,
and therefore $ N^{h-v(J)}\le N^{h-2-(\ell-1)/m_2(H)}$. Together with $d=\Theta_H(N^{h-2})$, this gives the required bound.
\end{proof}

\begin{proposition}[Containers in the dense random host]
\label{prop:containers}
With high probability, $F\sim G(N,1/2)$ has the following property.  For each
$0<\alpha<1$, there is a family $\calC\subseteq2^{E(F)}$ such that:
\begin{enumerate}[label=\textup{(\arabic*)}]
\item every independent set in $\calH_F$ is contained in some $C\in\calC$;
\item every $C\in\calC$ contains at most $\alpha e(\calH_F)$ hyperedges of
$\calH_F$;
\item
$\log |\calC|
        \le C_HN^{2-1/m_2(H)}\log N\log(1/\alpha).$
\end{enumerate}
Moreover, since $H$ has no isolated vertices, each hyperedge of $\calH_F$
corresponds to at most $h!$ labelled induced copies of $H$.
\end{proposition}

\begin{proof}
Apply Theorem~\ref{thm:containers} to $\calH_F$, using Lemma~\ref{lem:codegrees} and
$\tau=N^{-1/m_2(H)}$.  Since $|E(F)|=\Theta(N^2)$,
$$
        |E(F)|\tau\log(1/\tau)
        \le C_HN^{2-1/m_2(H)}\log N.
$$
This gives the stated bound on $\log|\calC|$.  Finally, if $H$ has no
isolated vertices, then the edge set of a copy determines its vertex set, and
there are at most $h!$ labelled embeddings on that vertex set.
\end{proof}

\section{Random thinning}

We now combine the supersaturation and container statements.

\begin{theorem}[Random thinning]
\label{thm:thinning}
Let $H$ be fixed with no isolated vertices and $e(H)\ge2$.  Put $\rho=2^{-25v(H)^2}r^{-4}$,
and let $M=\lceil r^{A_Hr}\rceil$, where $A_H$ is as in
Theorem~\ref{thm:acdfm-input}.  Let $F\sim G(N,1/2)$, and let $W\subseteq E(F)$ be
obtained by retaining every edge of $F$ independently with probability
$\theta$.  There is a constant $K_H>0$ such that, if 
$$
        \theta
        \ge K_H\rho MrN^{-1/m_2(H)}\log N\log(rM),
$$
then $(V(F),W)\longrightarrow_{\ind}(H)_r$ with probability tending to one whenever $N^2/M^4\to\infty$.
\end{theorem}

\begin{proof}
Condition on the conclusions of Corollary~\ref{cor:tuple-supersaturation} and
Proposition~\ref{prop:containers}.  The probability of their failure is
$o(1)$ under the assumption $N^2/M^4\to\infty$.

Choose $\alpha$ small enough that every container contains fewer than
$c_HN^h/(rM^h)$ labelled induced copies of $H$.  By
Proposition~\ref{prop:containers}, and using $e(\calH_F)=\Theta_H(N^h)$, it is enough to
take $\alpha=c'_H(rM^h)^{-1}$ for a sufficiently small constant $c'_H>0$.  Then
$$
        \log |\calC|
        \le C_HN^{2-1/m_2(H)}\log N\log(rM).
$$

Suppose that $(V(F),W)$ has an $r$-colouring with no monochromatic induced
copy of $H$.  Let $I_a\subseteq W$ be the set of edges of colour $a$.
Each $I_a$ is independent in $\calH_F$: if an induced copy of $H$ in
$F$ had all its graph-edges in $I_a$, then the same vertex set would span a
monochromatic induced copy of $H$ in $(V(F),W)$, since the non-edges of the
copy are already absent from $F$.  Hence $I_a\subseteq C_a$ for some
$C_a\in\calC$, and therefore $W\subseteq C_1\cup\cdots\cup C_r$.

By Corollary~\ref{cor:tuple-supersaturation}, every such tuple satisfies
$$
        |E(F)\setminus(C_1\cup\cdots\cup C_r)|
        \ge c_H\frac{N^2}{\rho M}.
$$
For a fixed tuple $(C_1,\ldots,C_r)$, the probability that the random
thinning selects no edge from this complement is at most
$$
        \exp\left\{-c_H\frac{\theta N^2}{\rho M}\right\}.
$$
Taking a union bound over $\calC^r$, the probability that a bad colouring
exists is at most
$$
        \exp\left\{
        C_HrN^{2-1/m_2(H)}\log N\log(rM)
        -c_H\frac{\theta N^2}{\rho M}
        \right\}.
$$
The assumed lower bound on $\theta$ makes the second term dominate the first,
after increasing $K_H$.  This proves the theorem.
\end{proof}

\begin{proof}[Proof of Theorem~\ref{thm:main}]
Fix $H$ and $\eta>0$.  Let $A_H$ and $K_H$ be as in
Theorem~\ref{thm:thinning}, and put $M=\lceil r^{A_Hr}\rceil$ and
$\rho=2^{-25v(H)^2}r^{-4}$.  Since a $\theta$-thinning of $G(N,1/2)$ has
distribution $G(N,\theta/2)$, we set $\theta=2p$.  The upper bound
$p\le1/2$ ensures $\theta\le1$.

By Theorem~\ref{thm:thinning}, it is enough that
$$
        p\ge K_H\rho MrN^{-1/m_2(H)}\log N\log(rM).
$$
Using $p\ge N^{-1/m_2(H)+\eta}$, it suffices that
$$
        N^\eta\ge K_H\rho Mr\log N\log(rM).
$$
Since $M=r^{O_H(r)}$ and $\rho r=2^{-25v(H)^2}r^{-3}$, this inequality
holds for all $N\ge r^{C_{H,\eta}r}$ if $C_{H,\eta}$ is chosen sufficiently
large.  Increasing $C_{H,\eta}$ if necessary also ensures $N^2/M^4\to
\infty$ in the range under consideration, so the error term in
Theorem~\ref{thm:thinning} tends to zero.  Bounded values of $r$ are absorbed into
the same constant.
\end{proof}

\begin{proof}[Proof of Corollary~\ref{cor:size}]
Choose $N=r^{C_{H,\eta}r}$ with $C_{H,\eta}$ sufficiently large and take
$$
        p=N^{-1/m_2(H)+\eta/2}.
$$
By Theorem~\ref{thm:main}, with high probability
$G(N,p)\to_{\ind}(H)_r$.  Chernoff's inequality gives, with high probability,
$$
        e(G(N,p))\le 2p\binom N2\le N^{2-1/m_2(H)+\eta}.
$$
Thus there exists an $N$-vertex induced-Ramsey graph with at most
$N^{2-1/m_2(H)+\eta}$ edges.  Rewriting this bound in terms of $r$ gives
the displayed estimate, after increasing $C_{H,\eta}$ once more if needed.
\end{proof}

\section{Concluding remarks}

The weighted input is used only through the pair-overlap estimate in
Lemma~\ref{lem:pair-overlap}.  If one used merely that a small dense graph is
induced-Ramsey, then the averaging argument would require $M$-vertex sets
containing no deleted edge.  This gives a complement lower bound of order
$N^2/M^2$ for bad tuples of containers.  The pair-overlap estimate instead
allows each good $M$-vertex subgraph to survive the deletion of
$\Theta(M/\rho)$ edges.  Consequently the complement lower bound becomes
$N^2/(\rho M)$, and the thinning threshold changes from
$M^2rN^{-1/m_2(H)}\log N\log(rM)$ to $\rho MrN^{-1/m_2(H)}\log N\log(rM)$.

The theorem is therefore best viewed as a transference statement from the
weighted dense random theorem of \cite{ACDFM}, rather than as a monotone
threshold theorem.  The restriction $p\le1/2$ is not a technical artefact: for
non-complete $H$, adding edges can destroy induced copies of $H$.

\textbf{Acknowledgments.}
The author is grateful to her supervisor, Imre Leader, for his support and guidance, and is supported by the CB European PhD Studentship funded by Trinity College, Cambridge.

\end{document}